\documentclass[11pt]{amsart}

\usepackage{amsmath,amssymb,amsthm}
\input xy  
\xyoption{all}
\newtheorem{theorem}{Theorem}[section]

\newtheorem{cor}[theorem]{Corollary}
\newtheorem{lemm}[theorem]{Lemma}
\newtheorem{prop}[theorem]{Proposition}
\newtheorem{question}[theorem]{Question}

\theoremstyle{definition}
\newtheorem{defi}[theorem]{Definition}
\newtheorem{remk}[theorem]{Remark}
\newtheorem{exam}[theorem]{Example}

\newcommand{\PP}{\mathcal{P}}

\newcommand{\aaa}{\mbox{\boldmath $a$}}
\newcommand{\bb}{\mbox{\boldmath $b$}}
\newcommand{\cc}{\mbox{\boldmath $c$}}
\newcommand{\dd}{\mbox{\boldmath $d$}}
\newcommand{\uu}{\mbox{\boldmath $u$}}
\newcommand{\vv}{\mbox{\boldmath $v$}}
\newcommand{\kk}{I}
\newcommand{\TT}{\operatorname{\mathcal T}\nolimits}

\newcommand{\K}{\mathsf{K}}
\newcommand{\thick}{\mathsf{thick}} 

\newcommand{\La}{\Lambda}

\newcommand{\point}[1]{\bullet}
\newcommand{\add}{\operatorname{add}\nolimits}
\newcommand{\proj}{\operatorname{proj}\nolimits}

\newcommand{\Hom}{\operatorname{Hom}\nolimits}

\newcommand{\RHom}{\mathbf{R}\strut\kern-.2em\operatorname{Hom}\nolimits}
\renewcommand{\mod}{\operatorname{mod}\nolimits}

\def\Ker{\mathop{\mathrm{Ker}}\nolimits}
\def\Coker{\mathop{\mathrm{Coker}}\nolimits}
\def\Hom{\mathop{\mathrm{Hom}}\nolimits}
\def\End{\mathop{\mathrm{End}}\nolimits}

\def\mod{\mathop{\mathrm{mod}}\nolimits}

\def\RHom{\mathop{\mathbb R\mathrm{Hom}}\nolimits}

\def\Soc{\mathop{\mathrm{Soc}}\nolimits}

\def\add{\mathop{\mathrm{add}}\nolimits}

\def\op{\mathop{\mathrm{op}}\nolimits}

\begin{document}
\title{On mutations of selfinjective quivers with potential}
\author{Yuya Mizuno}
\address{Graduate School of Mathematics\\ Nagoya University\\ Frocho\\ Chikusaku\\ Nagoya\\ 464-8602\\ Japan}
\email{yuya.mizuno@math.nagoya-u.ac.jp}
%\date{\today}c
\begin{abstract}
We study silting mutations (Okuyama-Rickard complexes) for selfinjective algebras given by quivers with potential (QPs). We show that silting mutation is compatible with QP mutation. 
As an application, we get a family of derived equivalences of Jacobian algebras.

\end{abstract}
\maketitle
%\newpage
%\tableofcontents
%\newpage
\section{Introduction}
Derived categories are nowadays considered as an essential tool in the study of many areas of mathematics. 
In the representation theory of algebras, derived equivalences of algebras
have been one of the central themes and extensively investigated. 
It is well-known that endomorphism algebras of tilting complexes are derived equivalent to the original algebra \cite{R1}.  
Therefore it is an important problem to give concrete methods to calculate endomorphism algebras of tilting complexes.   
In this paper, we focus on one of the fundamental tilting complexes over selfinjective algebras, known as Okuyama-Rickard complexes, 
which play an important role in the study of Brou\'e's abelian defect group conjecture. 
From a categorical viewpoint, they are nowadays
interpreted as a special case of silting mutation \cite{AI}. We provide a method to determine the quivers with relations of the endomorphism algebras of Okuyama-Rickard complexes when selfinjective algebras are given by quivers with potential (QPs for short).

The notion of QPs was introduced by \cite{DWZ}, which gives a better understanding of cluster algebras (we refer to \cite{K2}). 
Recently it has been discovered that mutations of QPs (Definition \ref{mutation QP}) 
give rise to derived equivalences in several situations, for example  \cite{BIRS,IR,KeY,L1,L2,M,V}. 
The deep connnection between mutations and derived equivalences is also quite useful to study the derived equivalence classification of cluster-tilted algebras \cite{B,BHL1,BHL2,BV}. 
The aim of this paper is to give a similar (but different) type of derived equivalences by comparing QP mutation and silting mutation (Definition \ref{silting mutation}). 

Our main result is the following (see sections 2 and 3 for unexplained notions).

\begin{theorem}\label{intro1}(Proposition \ref{tilt self}, Theorem \ref{main1}, Corollary \ref{cor} and Lemma \ref{leftright})
Let $(Q,W)$ be a selfinjective QP (Definition \ref{self QP}) and $\Lambda:=\PP(Q,W)$. 
For a set of vertices $\kk\subset Q_0$, we assume the following conditions. 
\begin{itemize}
\item[$\bullet$] Any vertex in $\kk$ is not contained in 2-cycles in $Q$.
\item[$\bullet$] There are no arrows between vertices in $\kk$.
\end{itemize}
 
{\rm (a)} We have an algebra isomorphism $$\End_{\K^{\rm{b}}(\proj{\Lambda})}(\mu_{\kk}(\Lambda))\cong\PP(\mu_{\kk}(Q,W)),$$
where $\mu_{\kk}(\Lambda)$ is left (or right) silting mutation and $\mu_{\kk}(Q,W)$ is a composition of  QP mutation of the vertices $\kk$.

{\rm (b)} If $\sigma I=I$ for the Nakayama permutation $\sigma$ of $\La$, then 
$\mu_{\kk}(\Lambda)$ is a tilting complex. In particular, 
$\La$ and $\PP(\mu_{\kk}(Q,W))$ are derived equivalent.

\end{theorem}
 
Since selfinjective algebras are closed under derived equivalence, we conclude that from (b) above the new QP is  also a selfinjective QP, which is a result given in \cite[Theorem 4.2]{HI}. 
Then we can apply our result to the new QP again and these processes provide a family of derived equivalences.  
We note that Keller-Yang \cite{KeY} proved that, for two QPs related by QP mutation, their Ginzburg dg algebras, which are certain enhancement of Jacobian algebras, are derived equivalent though their Jacobian algebras are far from being derived equivalent in general.
On the other hand, Theorem \ref{intro1} tells us that Jacobian algebras are already derived equivalent in our setting. 

\subsection*{Notations}
Let $K$ be an algebraically closed field and $D:=\Hom_K(-,K)$.
All modules are left modules.
For a finite dimensional algebra $\La$, we denote by $\mod\Lambda$ the category of finitely generated $\Lambda$-modules and by add$M$ the subcategory of $\mod\Lambda$ consisting of direct summands of finite direct sums of copies of $M\in \mod\La$.  
The composition $fg$ means first $f$, then $g$.
For a quiver $Q$, we denote by $Q_0$ vertices and $Q_1$ arrows of $Q$. We denote by $s(a)$ the start vertex and by $e(a)$ the end vertex of an arrow or path $a$.\\

\section{Preliminaries}
%%%%%%%%%%%%%%%%%%%%%%%%%%%%%%%%%%%%%%%%%%%%%%%%%%%%%%%%%%%%%%%%%%% 

\subsection{Quivers with potential}\label{QPs}
We recall the definition of quivers with potential. We follow \cite{DWZ}.

$\bullet$ Let $Q$ be a finite connected quiver without loops. 
We denote by $KQ_i$ the $K$-vector space with basis consisting of paths of length $i$ in $Q$, and by
$KQ_{i,cyc}$ the subspace of $KQ_i$ spanned by all cycles.
We denote the \emph{complete path algebra} by
$$\widehat{KQ}=\prod_{i\ge0}KQ_i$$
and by $J_{\widehat{KQ}}$ the Jacobson radical of $\widehat{KQ}$. 
A \emph{quiver with potential} (QP) is a pair $(Q,W)$ consisting of a finite connected quiver $Q$ without loops and 
an element $W\in\prod_{i\ge2}KQ_{i,{\rm cyc}}$, called a \emph{potential}.
For each arrow $a$ in $Q$, the \emph{cyclic derivative} $\partial_a:\widehat{KQ}_{cyc} \to \widehat{KQ}$  is defined as the continuous linear map 
satisfying $\partial_a(a_1\cdots a_d)=\sum_{a_i=a}a_{i+1}\cdots a_d a_1\cdots a_{i-1}$ for a cycle $a_1\cdots a_d$.
For a QP $(Q,W)$, we define the \emph{Jacobian algebra} by
\[\PP(Q,W)=\widehat{KQ}/{\mathcal J}(W),\]
where ${\mathcal J}(W)=\overline{\langle \partial_a W \mid a \in Q_1 \rangle}$
is the closure of the ideal generated by $\partial_aW$ with respect to the $J_{\widehat{KQ}}$-adic topology.

$\bullet$   
A QP $(Q,W)$ is called \emph{trivial} if $W$ is a linear combination
of cycles of length 2 and $\PP(Q,W)$ is isomorphic to the semisimple algebra $\widehat{KQ_0}$.
It is called \emph{reduced} if $W\in\prod_{i\ge3}KQ_{i,{\rm cyc}}.$

Following \cite{HI}, we use this terminology.

\begin{defi}\label{self QP}
We call a QP $(Q,W)$ \emph{selfinjective} if $\PP(Q,W)$ is a finite dimensional selfinjective algebra.
\end{defi}

Next we recall the definition of (pre-)mutation of QPs.

\begin{defi}\label{mutation QP} For each vertex $k$ in $Q$ not lying on a 2-cycle, we define \emph{pre-mutation} $\widetilde{\mu}_k(Q,W):=(Q',W')$ and obtain a new QP  as follows. 

\begin{itemize}
\item[(a)] $Q'$ is a quiver obtained from $Q$ by the following changes.

$\bullet$ Replace each arrow $a:u\to k$ in $Q$ by a new arrow $a^*:k \to u$.

$\bullet$ Replace each arrow $b:k\to v$ in $Q$ by a new arrow $b^*:v \to k$.

$\bullet$ For each pair of arrows $u\overset{a}{\to} k\overset{b}{\to}  v$, 
add a new arrow $[ab] : u \to v$
\item[(b)] $W' = [W] + \Delta$ is defined as follows.

$\bullet$  $[W]$ is obtained from the potential $W$ by replacing all compositions $ab$ by the new arrows $[ab]$ for each pair of arrows $u\overset{a}{\to} k\overset{b}{\to}  v$. 

$\bullet$  $\Delta={\displaystyle\sum_{\begin{smallmatrix}a,b\in Q_1\\
e(a)=k=s(b)\end{smallmatrix}}}[ba]a^*b^*$. 
\end{itemize}

Then \emph{mutation} ${\mu}_k(Q,W)$ is defined as a \emph{reduced part} of pre-mutation $\widetilde{\mu}_k(Q,W)$ (see \cite{DWZ}). 
\end{defi}

%%%%%%%%%%%%%%%%%%%%%%%%%%%%%%%%%%%%%%%%%%%%%%%%%%%%%%%%%%%%%%%%%%%%%%%%%%%%%%%%%

\subsection{Silting mutation}
The notion of silting objects was introduced by \cite{KV}, which is a generalization of tilting objects. 
Recently its theory has been rapidly developed and many connections have been discovered, for example \cite{BRT,AI,G,KoY}. 
In this section, we briefly recall their definitions and properties.

Let $\La$ be a finite dimensional algebra and 
$\TT:=\K^{\rm{b}}(\proj{\Lambda})$ be the homotopy category of bounded complexes of finitely generated projective $\Lambda$-modules. 

\begin{defi}
Let $T$ be an object of $\TT$. 
We call $T$ \emph{silting} (respectively, \emph{tilting})  
if $\Hom_{\TT}(T,T[i])=0$ for any positive integer $i>0$ (for any integer $i\neq0$) and satisfies $\TT=\thick{T}$, where $\thick T$ denotes  the smallest thick subcategory of $\TT$ containing $T$. 
\end{defi}

We call a morphism $f:X\to Y$ \emph{left minimal} 
if any morphism $g:Y\to Y$ satisfying $fg=f$ is an isomorphism. 
For an object $M\in\TT$, we call a morphism $f:X\to M'$ \emph{left $(\add{M})$-approximation} of $X$ 
if $M'$ belongs to $\add{M}$ and $\Hom_{\TT}(f,M'')$ is surjective for any object $M''$ in $\add{M}$. Dually we define a \emph{right minimal} morphism and a \emph{right $(\add{M})$-approximation}.

\begin{defi}\label{silting mutation}
Let $T$ be a basic silting object in $\TT$ and take an arbitrary decomposition $T=X\oplus M$. 
We take a minimal left $(\add{M})$-approximation $f:X\to M'$ of $X$ and a triangle 
\[\xymatrix{
X \ar[r]^{f} & M' \ar[r] & Y \ar[r] & X[1].
}\]
 
We put $\mu_{X}^{}(T):=Y\oplus M$ and call it a \emph{left silting mutation} of $T$ with respect to $X$. 
Dually we define a \emph{right silting mutation}. 
\end{defi}

We recall an important result of silting mutation.

\begin{theorem}\cite[Theorem 2.31]{AI}\label{mutation is silting}
Any mutation of a silting object is again a silting object.
\end{theorem}

Then we apply the above theory  to the following situation.

Let $Q$ be a finite connected quiver and $\La:=\widehat{KQ}/\overline{\langle R\rangle}$ be a complete finite dimensional algebra.
We denote by $\{e_k\ |\ k\in Q_0\}$ a complete set of primitive orthogonal idempotents of $\La$. 
Take a set of vertices $\kk:=\{k_1,\ldots,k_n\}\subset Q_0$ and we denote by $e_{\kk}:=e_{k_1}+\dots+e_{k_n}$. Then we define a so-called Okuyama-Rickard complex
$$\mu_{\kk}(\Lambda):=\mu_{\La e_{\kk}}(\Lambda)=\left\{\begin{array}{ccc}
\stackrel{-1}{\La e_I}&\stackrel{f}{\longrightarrow}&\stackrel{0}{ \La e'}\\
&&\oplus\\
&&\La(1-e_I).
\end{array}\right.$$ 
Here $f$ is a minimal left $(\add \La(1-e_I))$-approximation of $\La e_I$.

By Theorem \ref{mutation is silting}, $\mu_{\kk}(\Lambda)$ is always a silting object of $\TT$, but it is not necessarily a tilting object. 
However, for selfinjective algebras, 
it is a tilting object if it satisfies a condition given by Nakayama permutations.

\begin{defi}
Let $\La$ be a selfinjective algebra. 
Then there exists a permutation $\sigma:Q_0\to Q_0$ satisfying  
$D(e_k\Lambda)\cong\Lambda e_{\sigma (k)}$ for any $k\in Q_0$. 
We call $\sigma$ the \emph{Nakayama permutation} of $\La$. 
\end{defi}

Note that $\La e_{\kk}\cong\nu(\La e_{\kk})$ if and only if $\kk=\sigma\kk$, where $\nu:= D\Hom_{\La}(-,\La):\mod\La\to\mod\La$ is the Nakayama functor. 
The following easy result is quite useful. We refer to \cite{AH,AI,D} for the proof.

\begin{prop}\label{tilt self}
Let $\La$ be a selfinjective algebra above. 
Then $\mu_{\kk}(\Lambda)$ is a tilting object in $\TT$ if and only if $\kk=\sigma\kk$. 
\end{prop}

%%%%%%%%%%%%%%%%%%%%%%%%%%%%%%%%%%%%%%%%%%%%%%%%%%%%%%%%%%%%%%%%%%%%%%%%%%%%%%%%%%%%%%%%%%%%%%
\section{Main results}\label{main section}

For a set of vertices $\kk:=\{k_1,\ldots,k_n\}\subset Q_0$, we assume the following conditions. 
\begin{itemize}
\item[(a1)] Any vertex in $\kk$ is not contained in 2-cycles in $Q$.
\item[(a2)] There are no arrows between vertices in $\kk$.
\end{itemize}

Since a mutation of QPs is given by changing the neighboring arrows associated with a vertex (Definition \ref{mutation QP}), 
the composition of the (pre-)mutation at the vertices in $\kk$ is independent of the choice of the order of mutations in this case. 
Hence we can define the successive (pre-)mutation as follows
\[\widetilde{\mu}_{\kk}(Q,W):=\widetilde{\mu}_{k_1}\circ\cdots\circ\widetilde{\mu}_{k_n}(Q,W),\] \[{\mu}_{{\kk}}(Q,W):={\mu}_{k_1}\circ\cdots\circ{\mu}_{k_n}(Q,W).\]

Then our main result is the following.

\begin{theorem}\label{main1}
Let $(Q,W)$ be a selfinjective QP and $\Lambda:=\PP(Q,W)$. 
Let $\kk$ be a set of vertices of $Q_0$ satisfying the conditions {\rm (a1)} and {\rm (a2)}.
Then we have a $K$-algebra isomorphism $$\End_{\K^{\rm{b}}(\proj\La)}(\mu_{\kk}(\Lambda))\cong\PP(\mu_{\kk}(Q,W)).$$
\end{theorem}

We will give the proof in the next section.
Combining this result with Theorem \ref{tilt self}, we have the following result.

\begin{cor}\label{cor}
Let $\kk$ be a set of vertices of $Q_0$ satisfying $\sigma\kk=\kk$ and the conditions {\rm (a1)} and {\rm (a2)}.
Then $\PP(Q,W)$ and $\PP(\mu_{\kk}(Q,W))$ are derived equivalent.
\end{cor}

\begin{proof}
By Proposition \ref{tilt self}, $\mu_{\kk}(\Lambda)$ is a tilting object of $\K^{\rm{b}}(\proj\La)$. 
Then $\La=\PP(Q,W)$ and $\End_{\K^{\rm{b}}(\proj\La)}(\mu_{\kk}(\Lambda))$ are derive equivalent by \cite{R1}. 
On the other hand, Theorem \ref{main1} implies $\End_{\K^{\rm{b}}(\proj\La)}(\mu_{\kk}(\Lambda))\cong\PP(\mu_{\kk}(Q,W))$ and the statement follows.
\end{proof}

%We remark that the statements above are far from true if $\La$ is not selfinjective. 
Moreover, since selfinjectivity is preserved by derived equivalence \cite{AR}, 
we have the following result, which is given in \cite[Theorem 4.2]{HI}.

\begin{cor}Let $\kk$ be a set of vertices of $Q_0$ satisfying $\sigma\kk=\kk$ and the conditions {\rm (a1)} and {\rm (a2)}.
Then $\mu_{\kk}(Q,W)$ is a selfinjective QP.
\end{cor}
 
We note that the Nakayama permutation of $\mu_{\kk}(Q,W)$ is again given by the same permutation \cite[Proposition 4.4.(b)]{HI}. By this corollary, we can apply Corollary \ref{cor} to the new QPs repeatedly. %and, consequently, obtain a lot of derived equivalences.

We considered only left mutation, but 
the following lemma shows that the same result holds for right mutation.

\begin{lemm}\label{leftright}
Under the assumption of Theorem \ref{main1}, 
we have a $K$-algebra isomorphism 
$$\End_{\K^{\rm{b}}(\proj\La)}(\mu_{\kk}(\Lambda))\cong\End_{\K^{\rm{b}}(\proj\La)}(\mu_{\kk}'(\Lambda)),$$
where $\mu_{\kk}'(\Lambda)$ is a right mutation of $\La$.

\end{lemm}

\begin{proof}

First, note that we have $\La^{\op}\cong\PP(Q,W)^{\op}\cong\PP(Q^{\op},W^{\op}),$ where $W^{\op}$ is the corresponding potential of $Q^{\op}$ to $W$. 
Then it is clear that $\PP(\mu_{\kk}(Q^{\op},W^{\op}))\cong\PP(\mu_{\kk}(Q,W))^{\op}$ holds.

On the other hand, 
we have a duality $$\Hom_\La(-,\La):\K^{\rm{b}}(\proj\La)\overset{\simeq}{\to}\K^{\rm{b}}(\proj\La^{\op}),$$
which sends $\mu_{\kk}'(\Lambda)$ to $\mu_{\kk}(\Lambda^{\op})$.
Thus, 
we have 
\begin{eqnarray*}
\End_{\K^{\rm{b}}(\proj\La)}(\mu_{\kk}'(\Lambda))&\cong&(\End_{\K^{\rm{b}}(\proj\La^{\op})}(\mu_{\kk}(\Lambda^{\op})))^{\op}\\
&\cong&(\PP(\mu_{\kk}(Q^{\op},W^{\op})))^{\op}\ \ \ \  ({\rm Theorem\ \ref{main1}\ to}\ \La^{\op})\\
&\cong&\PP(\mu_{\kk}(Q,W)). 
\end{eqnarray*} 
\end{proof}

By this result, we consider only left mutations in this paper.

\begin{exam}\label{exam1}
Let $(Q,W)$ be the QP given as follows

$$\xymatrix@C10pt@R10pt{
 & & 1 \ar[ld]_{a_1}& & \\
 & 2  \ar[ld]_{a_2} & & 6\ar[lu]_{a_6} & \\
3 \ar[rr]|{a_3}& & 4 \ar[rr]|{a_4}  & & 5,\ar[lu]_{a_5}&W=a_1a_2a_3a_4a_5 a_6.}$$

Then $(Q,W)$ is a selfinjective QP with a Nakayama permutation $(153)(264)$.
Let $\La:=\PP(Q,W)$ and $\TT:=\K^{\rm{b}}(\proj\La)$ and take a silting object in $\TT$
$$\mu_1(\La)=\left\{\begin{array}{ccc}
\stackrel{-1}{\La e_1}&\stackrel{a_1}{\longrightarrow}&\stackrel{0}{\La e_2}\\
&&\oplus\\
&&\La(1-e_1).
\end{array}\right.$$ 
Note that $\mu_1(\La)$ is not a tilting object. 
By Theorem \ref{main1}, we have an isomorphism 
$$\End_{\TT}(\mu_{1}(\Lambda))\cong\PP(\mu_{1}(Q,W)),$$
where $\mu_{1}(Q,W)$ is the QP given as follows

$$\xymatrix@C10pt@R10pt{
 & & 1 \ar[rd]^{a_6^*}& & \\
 & 2 \ar[ru]^{a_1^*} \ar[ld]_{a_2} & & 6 \ar[ll]^{[a_6a_1]}& \\
3 \ar[rr]|{a_3}& & 4 \ar[rr]|{a_4}  & & 5,\ar[lu]_{a_5}&[a_6a_1]a_1^*a_6^*+[a_6a_1]a_2a_3a_4a_5.}$$

On the other hand, we consider the $\sigma$-orbit of the vertex 1 and let $\kk=\{1,3,5\}$. 
Then we have a tilting object 
$$
\mu_{\kk}(\La)=\left\{\begin{array}{ccc}
\stackrel{-1}{\La e_1\oplus\La e_3\oplus\La e_5}&\stackrel{\left(\begin{smallmatrix}a_1&0&0\\ 
0&a_3&0\\
0&0&a_5
\end{smallmatrix}\right)}{\longrightarrow}&\stackrel{0}{\La e_2\oplus\La e_4\oplus\La e_6}\\
&&\oplus\\
&&\La(1-e_{\kk}).
\end{array}\right.
$$ 

Then we have an isomorphism 
$$\End_{\TT}(\mu_{\kk}(\Lambda))\cong\PP(\mu_{\kk}(Q,W)),$$
where $\mu_{\kk}(Q,W)$ is the QP given as follows

$$
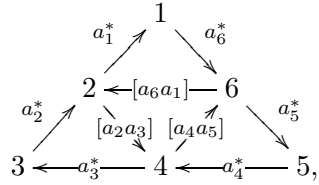
\xymatrix@C15pt@R15pt{
 & & 1 \ar[rd]^{a_6^*}& && \\
 & 2\ar[ru]^{a_1^*}  \ar[rd]|{[a_2a_3]}& & 6\ar[rd]^{a_5^*} \ar[ll]|{[a_6a_1]}&& \\
3 \ar[ru]^{a_2^*}& & 4 \ar[ll]|{a_3^*} \ar[ru]|{[a_4a_5]} & & 5,\ar[ll]|{a_4^*} 
}$$
$[a_6a_1]a_1^*a_6^*+[a_2a_3]a_3^*a_2^*+[a_4a_5]a_5^*a_4^*
+[a_6a_1][a_2a_3][a_4a_5].$

We note that, although $\PP(\mu_{\kk}(Q,W))$ is selfinjective and derived equivalent to $\PP(Q,W)$, 
 $\PP(\mu_{1}(Q,W))$ is neither selfinjective nor derived equivalent to $\PP(Q,W)$.
\end{exam}

\begin{remk}
We remark that the above algebras are cluster-tilted algebras of type $D$. 
Derived equivalences of the algebras of this type are extensively investigated by Bastian-Holm-Ladkani \cite{BHL1}. 
In particular, derived equivalences of the above algebras is given in \cite[Example 2.19]{BHL1} (see Section 4.2 and Lemma 4.5 in that paper for the general case). 
These derived equivalences follow also from Asashiba's derived equivalence classification of selfinjective algebras of finite representation type \cite{As}. 
\end{remk}

\begin{exam} Let $(Q,W)$ be the QP given as follows  
$$
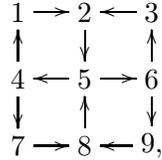
\begin{xy} 0;<0.2pt,0pt>:<0pt,-0.2pt>:: 
(0,0) *+{ 1} ="1",
(126,0) *+{ 2} ="2",
(253,0) *+{ 3} ="3",
(0,126) *+{ 4} ="4",
(126,126) *+{5} ="5",
(253,126) *+{6} ="6",
(0,253) *+{7} ="7",
(126,253) *+{8} ="8",
(253,253) *+{9,} ="9",
"1", {\ar"2"},
"4", {\ar"1"},
"3", {\ar"2"},
"2", {\ar"5"},
"6", {\ar"3"},
"5", {\ar"4"},
"4", {\ar"7"},
"5", {\ar"6"},
"8", {\ar"5"},
"6", {\ar"9"},
"7", {\ar"8"},
"9", {\ar"8"},
\end{xy}$$
where the potential is the sum of each small squares.
Then $(Q,W)$ is a selfinjective QP with a Nakayama permutation $(19)(28)(37)(46)(5).$ 
For $\sigma$-orbits $\kk_1:=\{1,9\}$ and $\kk_3:=\{3,7\}$,   
we have selfinjective QPs $\mu_{\kk_1}(Q,W)$ and $\mu_{\kk_3}\circ\mu_{\kk_1}(Q,W)$ and their  Jacobian algebras are derived equivalent to $\PP(Q,W)$. %(we omit the potential).
$$
\begin{xy} 0;<0.2pt,0pt>:<0pt,-0.2pt>:: 
(0,0) *+{ 1} ="1",
(126,0) *+{2} ="2",
(253,0) *+{3} ="3",
(0,126) *+{4} ="4",
(-200,126) *+{\overset{\mu_{\kk_1}(Q,W)}{\longrightarrow}},
(126,126) *+{5} ="5",
(253,126) *+{6} ="6",
(0,253) *+{7} ="7",
(126,253) *+{8} ="8",
(253,253) *+{9,} ="9",
"2", {\ar"1"},
"1", {\ar"4"},
"3", {\ar"2"},
"4", {\ar"2"},
"2", {\ar"5"},
"6", {\ar"3"},
"5", {\ar"4"},
"4", {\ar"7"},
"5", {\ar"6"},
"8", {\ar"5"},
"6", {\ar"8"},
"9", {\ar"6"},
"7", {\ar"8"},
"8", {\ar"9"},
\end{xy}\ \ \ \  
\begin{xy} 0;<0.2pt,0pt>:<0pt,-0.2pt>:: 
(0,0) *+{ 1} ="1",
(126,0) *+{2} ="2",
(253,0) *+{3} ="3",
(0,126) *+{4} ="4",
(-200,126) *+{\overset{\mu_{\kk_3}\circ\mu_{\kk_1}(Q,W)}{\longrightarrow}},
(126,126) *+{5}="5",
(253,126) *+{6} ="6",
(0,253) *+{7} ="7",
(126,253) *+{8} ="8",
(253,253) *+{9.} ="9",
"2", {\ar"1"},
"1", {\ar"4"},
"2", {\ar"3"},
"4", {\ar"2"},
"2", {\ar"5"},
"3", {\ar"6"},
"5", {\ar"4"},
"7", {\ar"4"},
"5", {\ar"6"},
"8", {\ar"5"},
"6", {\ar"8"},
"9", {\ar"6"},
"8", {\ar"7"},
"8", {\ar"9"},
"6", {\ar"2"},
"4", {\ar"8"},
\end{xy}$$
\end{exam}

\begin{exam}
Let $(Q,W)$ be the QP associated with tubular algebra of type $(2,2,2,2)$

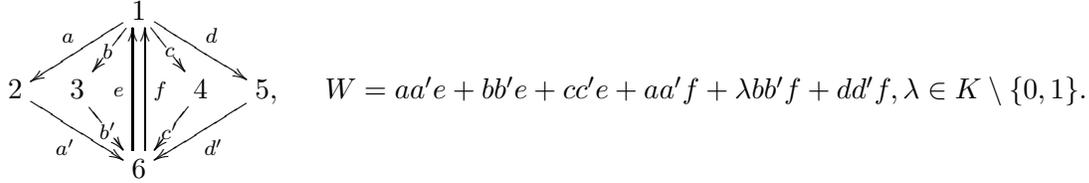
\[\xymatrix@R=.5cm@C=.4cm{
&&1\ar_a[dll]\ar|b[dl]\ar|c[dr]\ar^d[drr]\\
2\ar_{a'}[drr]&3\ar|{b'}[dr]&&4\ar|{c'}[dl]&5,\ar^{d'}[dll]&  W=aa'e+bb'e+cc'e+aa'f+\lambda bb'f+dd'f, \lambda\in K\setminus\{0,1\}.\\
&&6\ar@<.5ex>^e[uu]\ar@<-.5ex>_f[uu]
}\]

This arises naturally in the context of weighted projective line \cite{GL}.
Then $(Q,W)$ is a selfinjective QP \cite{HI} and the Nakayama permutation is the identity. 
Thus mutation of the QP at any vertex yields a derived equivalence in this case.
For example, $\mu_2(Q,W)$ is the following QP with $\lambda'=\frac{\lambda}{\lambda-1}$
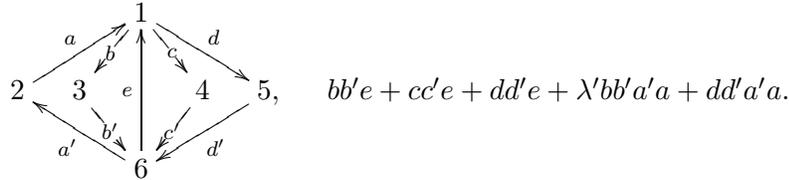
\[\xymatrix@R=.5cm@C=.4cm{
&&1\ar|b[dl]\ar|c[dr]\ar^d[drr]\\
2\ar^{a}[urr]&3\ar|{b'}[dr]&&4\ar|{c'}[dl]&5,\ar^{d'}[dll]&bb'e+cc'e+dd'e+\lambda'bb'a'a+dd'a'a.\\
&&6\ar^{a'}[ull]\ar^e[uu]
}\]
Thus $\mu_2(Q,W)$ is a selfinjective QP and $\PP(\mu_2(Q,W))$ is derived equivalent to $\PP(Q,W)$.

\end{exam}

\begin{exam} Let $(Q,W)$ be the QP given as follows

$$
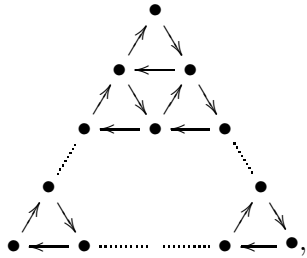
\begin{xy} 
<0pt,0pt>;<0.4pt,0pt>:<0pt,-0.4pt>:: 
(134,0) *+{\bullet} ="0",
(100,56) *+{\bullet} ="1",
(167,56) *+{\bullet} ="2",
(67,112) *+{\bullet} ="3",
(134,112) *+{\bullet} ="4",
(200,112) *+{\bullet} ="5",
(33,167) *+{\bullet} ="6",
(100,167) *+{} ="7",
(134,167) *+{} ="15",
(167,167) *+{} ="8",
(234,167) *+{\bullet} ="9",
(0,223) *+{\bullet} ="10",
(67,223) *+{\bullet} ="11",
(134,223) *+{} ="12",
(200,223) *+{\bullet} ="13",
(267,223) *+{\bullet,} ="14",
"1", {\ar"0"},
"0", {\ar"2"},
"2", {\ar"1"},
"3", {\ar"1"},
"1", {\ar"4"},
"4", {\ar"2"},
"2", {\ar"5"},
"4", {\ar"3"},
"6", {\ar@{.}"3"},
%"3", {\ar"7"},
"5", {\ar"4"},
%"7", {\ar"4"},
%"4", {\ar"8"},
%"8", {\ar"5"},
"5", {\ar@{.}"9"},
%"7", {\ar"6"},
"10", {\ar"6"},
"6", {\ar"11"},
%"8", {\ar"7"},
%"11", {\ar"7"},
%"7", {\ar"12"},
%"9", {\ar"8"},
%"12", {\ar"8"},
%"8", {\ar"13"},
"13", {\ar"9"},
"9", {\ar"14"},
"11", {\ar"10"},
"12", {\ar@{.}"11"},
"13", {\ar@{.}"12"},
"14", {\ar"13"},
%"12", {\ar@{.}"4"},

\end{xy}
$$
where the potential is the sum of  small triangles. 
Then $(Q,W)$ is a selfinjective QP and one can easily get a lot of derived equivalent algebras by the same procedures as above. 
See \cite[Figure 4]{HI} for one of the concrete description. We refer to \cite{K1}, which enables one to compute quiver mutations immediately. 
\end{exam}

Thus, from a given selfinjective Jacobian algebra, 
QP mutations give new selfinjective Jacobian algebras which are derived equivalent to the original one. 
Here we give a natural question that we find important for better understanding
of selfinjective QPs and derived equivalences.

\begin{question}
Let $\Lambda$ and $\Gamma$ be derived equivalent selfinjective
algebras. Then $\Lambda$ is isomorphic to a Jacobian algebra of a QP
if and only if so is $\Gamma$.
\end{question}

%%%%%%%%%%%%%%%%%%%%%%%%%%%%%%%%%%%%%%%%%%%%%%%%%%%%%%%%%%%%%%%%%%%%%%%%%%%%%%%%%%%%%%%%%%%%%%

\section{Proof of main result}
The basic strategy of the proof of our result is similar to the method given in \cite{BIRS}. 
Roughly speaking, we use the fact that, for a basic finite dimensional algebra $\La$, minimal projective presentations of simple $\La$-modules determine the quiver with relations of $\La$.  
We start with recalling results there.

\subsection{Presentation of algebras}
Let $\TT:=\K^{\rm{b}}(\proj\Lambda)$ for a finite dimensional algebras $\La$ and $J_{\TT}$ be the Jacobson radical of $\TT$.

\begin{defi}
Take an object $T\in\TT$. 
We call a complex
\[U\xrightarrow{f_1}V\xrightarrow{f_0}X\] 
a \emph{right 2-almost split sequence in $\add T$} if
\begin{eqnarray*}
\TT(T,U)\xrightarrow{f_1}\TT(T,V)\xrightarrow{f_0}J_{\TT}(T,X)\to0
\end{eqnarray*}
is exact. In other words, $f_0$ is right almost split in $\add T$ and $f_1$ is a pseudo-kernel of $f_0$ in $\add T$. 
Dually, we call a complex
\[X\xrightarrow{f_2}U\xrightarrow{f_1}V\] 
a \emph{left 2-almost split sequence in $\add T$} if
\begin{eqnarray*}
\TT(V,T)\xrightarrow{f_1}\TT(U,T)\xrightarrow{f_2}J_{\TT}(X,T)\to0
\end{eqnarray*}
is exact. In other words, $f_2$ is left almost split in $\add T$ and $f_1$ is a pseudo-cokernel of $f_2$ in $\add T$. 
We call a complex
\[X\xrightarrow{f_2}U\xrightarrow{f_1}V\xrightarrow{f_0}X\]
 a \emph{weak 2-almost split sequence in $\add T$} if
$U\xrightarrow{f_1}V\xrightarrow{f_0}X$ is a right 2-almost split sequence and
$X\xrightarrow{f_2}U\xrightarrow{f_1}V$ is a left 2-almost split sequence.
\end{defi}

Let $Q$ be a finite connected quiver. For $a\in Q_1$, define a \emph{right derivative} $\partial^{\mathbb{R}}_a: J_{\widehat{KQ}}\to\widehat{KQ}$ by 
\begin{eqnarray*}
&\partial^\mathbb{R}_a(a_1a_2\cdots a_{m-1}a_m)=
\left\{\begin{array}{cc}
a_1a_2\cdots a_{m-1}&\mbox{ if }\ a_m=a,\\
0&\mbox{ otherwise,}
\end{array}\right.&
\end{eqnarray*}
and extend to $J_{\widehat{KQ}}$ linearly and continuously.

We call an element of $\widehat{KQ}$ \emph{basic} if it is a formal linear sum of paths in $Q$ with a common start and a common end. Then we have the following result.

\begin{prop}\label{birs2}\cite[Section 3]{BIRS}
Let $Q$ be a finite connected quiver and $\Gamma$ be a basic finite dimensional algebra.
Let $\phi: \widehat{KQ} \to \Gamma$ be an algebra homomorphism and 
$R$ be a finite set of basic elements in $J_{\widehat{KQ}}$. 
Then the following conditions are equivalent.
\begin{itemize}
\item[(a)]$\phi$ is surjective and $\Ker\phi=\overline{S}$ for the ideal $S=\langle R\rangle$ of ${\widehat{KQ}}$, where $\overline{(\ )}$ denote the closure. 
\item[(b)]
The following sequence is exact for any $i\in Q_0$.
\[\xymatrix@C40pt{  {\displaystyle\bigoplus_{\begin{smallmatrix}r\in R, e(r)=i\end{smallmatrix}}}\Gamma(\phi s(r)) \ar[r]^{{}_r(\phi\partial_a^\mathbb{R}r)_{a}}  & 
{\displaystyle\bigoplus_{\begin{smallmatrix}a\in Q_1, e(a)=i\end{smallmatrix}}} \Gamma(\phi s(a)) \ar[r]^{\ \ \ \ \  \ {}_a(\phi a)} &  J_\Gamma(\displaystyle{\phi i}) \ar[r]& 0.  }\]
\end{itemize}
\end{prop}

\subsection{Our settings}
%We give some notations for our proof. 
We keep the assumption of Theorem \ref{main1}, that is, let $(Q,W)$ be a selfinjective QP,   $\Lambda:=\PP(Q,W)$ and $\kk$ a set of vertices of $Q_0$ satisfying the conditions {\rm (a1)} and {\rm (a2)}.
We denote by $P_i$ the indecomposable projective $\La$-module corresponding to the vertex $i \in Q_0$ and let $\TT:=\K^{\rm{b}}(\proj\La)$. 
Without loss of generality, we can assume that $(Q,W)$ is reduced since $\PP(\mu_\kk(Q,W))$
is isomorphic to $\PP(\mu_\kk(Q_{\rm red},W_{\rm red}))$, where $(Q_{\rm red},W_{\rm red})$ is a reduced part of $(Q,W)$ (\cite{DWZ}).

For a pair of arrows $a$ and $b$, define $\partial_{(a,b)}W$ by
\[\partial_{(a,b)}(a_1a_2\cdots a_m)=\sum_{a_i=a,\ a_{i+1}=b}a_{i+2}\cdots a_ma_1\cdots a_{i-1}\]
for any cycle $a_1\cdots a_m$ in $W$ and extend linearly and continuously. 
We denote by $\phi$ the natural surjective map $\widehat{KQ}\to \PP(Q,W)$. 
We simply denote $\phi p$ by $p$ for any element $p$ in $\widehat{KQ}$. 
Then, for any $i\in Q_0$, we have the following exact sequence 
in $\mod\Lambda$ \cite[Theorem 3.7]{HI}.

\begin{eqnarray}\label{resolution}
{\displaystyle
P_i \xrightarrow{f_{i2}:=(b)_b} \overbrace{\bigoplus_{\begin{smallmatrix}b\in Q_1, s(b)=i\end{smallmatrix}} P_{e(b)}}^{U_{i}:=} 
\xrightarrow{f_{i1}:={}_b(\partial_{(a,b)}W)_{a}}  \overbrace{\bigoplus_{\begin{smallmatrix}a\in Q_1, e(a)=i\end{smallmatrix}} P_{s(a)}}^{V_{i}:=} \xrightarrow{f_{i0}:={}_a(a)} 
P_i.}
\end{eqnarray}

Note that $f_{i2}$ is a minimal left $(\add(\Lambda/P_i))$-approximation and 
$f_{i0}$ is a minimal right $(\add(\Lambda/P_i))$-approximation. 
We embed the morphism $f_{i2}$  to a triangle in $\TT$ 
\begin{eqnarray}\label{h_i}
P_i \xrightarrow{f_{i2}}U_{i}\xrightarrow{h_i}P_i^*\xrightarrow{} P_i[1].
\end{eqnarray}
 
Then $P_i^*$ is the object  
$$(\cdots\to0\to \stackrel{-1}P_i\xrightarrow{f_{i2}}\stackrel{0}U_{i}\to0\to\cdots)$$ 
and we have a complex
\begin{eqnarray}\label{g_i}
P_i^* \xrightarrow{g_i}V_{i}\xrightarrow{f_{i0}}P_i
\end{eqnarray}
in $\TT$, where $g_i=({g_i}^j)_{j\in\mathbb{Z}}$ is defined by $\ {g_i}^0=f_{i1}$ and ${g_i}^j=0$ for $j\neq0$. 
%We remark that this complex is not a triangle.

Using these notations, $\mu_I(\La)$ is given as follows
$$\mu_I(\La)=\bigoplus_{\begin{smallmatrix}l\in I\end{smallmatrix}}P_l^*\oplus\bigoplus_{\begin{smallmatrix}j\notin I\end{smallmatrix}}P_j.$$ 
In the  rest of this paper, we put $T:=\mu_I(\La)$ for simplicity.

Since $\kk$ satisfies the conditions {\rm (a1)} and {\rm (a2)}, we can consider a pre-mutation $\widetilde{\mu}_{\kk}(Q,W)$. Let $(Q',W'):=\widetilde{\mu}_{\kk}(Q,W)$. Then $W'=[W]+\Delta$ is defined as follows. 

$\bullet$  $[W]$ is obtained from the potential $W$ by replacing all compositions $ab$ by the new arrows $[ab]$ for each pair of arrows $u\overset{a}{\to} l\overset{b}{\to}  v$ and $l\in\kk$. 

$\bullet$  $\Delta={\displaystyle\sum_{\begin{smallmatrix}a,b\in Q_1,l\in\kk\\
e(a)=l=s(b)\end{smallmatrix}}}[ab]b^*a^*$. 

Then, we define a $K$-algebra homomorphism  $\phi':\widehat{KQ'} \to \End_{\TT}(T)$ as follows. 

\begin{itemize}

\item[$\bullet$]For $l\in I$, define $\phi'l={\bf p}_l{\bf i}_l\in\End_\La(T)$, where ${\bf p}_l$ is the canonical projection ${\bf p}_l:T\to  P_l^*$ and ${\bf i}_l$ is the canonical injection ${\bf p}_l:P_l^*\to T$. For $j\in Q_0'$ with $j\notin I$, define $\phi'j=\phi j$.
 
\item[$\bullet$]For $a\in Q_1\cap Q_1'$, define $\phi'a=\phi a$.
\item[$\bullet$]Define $\phi'[ab]=\phi a\phi b$  for each pair of arrows $u\overset{a}{\to} l\overset{b}{\to}  v$.

Define  
% $\phi'a^{\ast}$ ($e(a)=l$) and $\phi'b^{\ast}$ ($s(b)=l$) are defined by the equalities
\begin{eqnarray*}
\phi'((a^{\ast})_{a\in Q_1,\ e(a)=l})=g_l&\in&\TT(P_l^{\ast},\bigoplus_{a\in Q_1,\ e(a)=l}P_{s(a)}),\\
\phi'((b^{\ast})_{b\in Q_1,\ s(b)=l})=-h_l&\in&\TT(\bigoplus_{b\in Q_1,\ s(b)=l}P_{e(b)},P_l^{\ast}).
\end{eqnarray*}
where $g_l,h_l$ are given in (\ref{h_i}), (\ref{g_i}).
\end{itemize}

As before, we simply denote $\phi' p$ by $p$ for any element $p$ in $\widehat{KQ'}$. 
Then we give the following proposition.

\begin{prop}\label{ver l}
For a vertex $l$ in $Q_0$ with $l\in\kk$, 
we have the following right 2-almost split sequence in $\add T$
\[{\bigoplus_{a\in Q_1,\ e(a)=l}P_{s(a)}\xrightarrow{{}_a([ab])_{b}}
\bigoplus_{b\in Q_1,\ s(b)=l}P_{e(b)}\xrightarrow{{}_b(b^{\ast})}P_l^{\ast}.}\]
\end{prop}

Next take a vertex $j\in Q_0$ with $j\notin \kk$.
Let  
$\kk_1:=\{l\in\kk\ |\ \exists u\in Q_1; s(u)=l,e(u)=j \}$ and $\kk_2:=\{l\in\kk\ |\ \exists v\in Q_1; s(v)=j,e(v)=l \}$.

We define complexes by 
\begin{eqnarray*}(P_{\kk_1}^* \xrightarrow{\aaa^*}V_{{\kk_1}}\xrightarrow{\aaa}P_{\kk_1}):=\bigoplus_{\begin{smallmatrix}k_1\in\kk_1\end{smallmatrix}}(P_{k_1}^* \xrightarrow{(a^*)_a}V_{{k_1}}\xrightarrow{{}_a(a)}P_{k_1}),\\
(P_{\kk_2} \xrightarrow{\bb}U_{{\kk_2}}\xrightarrow{\bb^*}P_{\kk_2}^*):=\bigoplus_{\begin{smallmatrix}k_2\in\kk_2\end{smallmatrix}}( P_{k_2} \xrightarrow{(b)_b}U_{{k_2}}\xrightarrow{{}_b(b^*)}P_{k_2}^*).
\end{eqnarray*}

Moreover we decompose $V_{j}=P_{\kk_1}\oplus V_{j}'$ and  $U_{j}= P_{\kk_2}\oplus U_{j}'$.
Then we can write the sequence (\ref{resolution}) by 

\begin{eqnarray*}
P_j\xrightarrow{(\dd\ \vv)}U_{j}'\oplus P_{\kk_2}\xrightarrow{{f_1 \aaa\ \ \ f_1' \choose \bb f_2 \aaa\ \ \ \bb f_2'}}P_{\kk_1}\oplus V_{j}'\xrightarrow{{\uu\choose \cc}}P_j,
\end{eqnarray*}
where 
\begin{eqnarray*}
\uu:={}_u(u)\ {\rm for}\ \{u\in Q_1\ |\ s(u)\in\kk,e(u)=j\},\\
\cc:={}_c(c)\ {\rm for}\ \{c\in Q_1\ |\ s(c)\notin\kk,e(c)=j\},\\
\vv:=(v)_v\ {\rm for}\ \{v\in Q_1\ |\ s(v)=j,e(v)\in\kk \},\\
\dd:=(d)_d\ {\rm for}\ \{d\in Q_1\ |\ s(d)=j,e(d)\notin\kk\}, \\
f_1:={}_d(\partial_{([au],d)}[W])_{(a,u)},\ \ \ 
f_1':={}_d(\partial_{(c,d)}[W])_{c},\\
f_2:={}_{(v,b)}(\partial_{([au],[vb])}[W])_{(a,u)},\ \ \ 
f_2':={}_{(v,b)}(\partial_{(c,[vb])}[W])_{c}.
\end{eqnarray*}

Then we have the following diagram

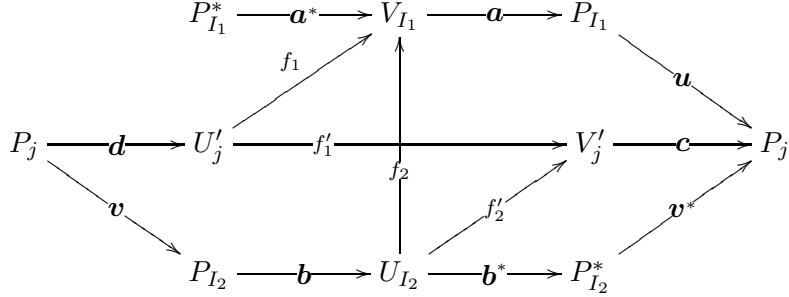
\begin{eqnarray*}\label{dia}\xymatrix@C50pt@R30pt{
&P_{\kk_1}^* \ar[r]|{\aaa^*}  &V_{\kk_1} \ar[r]|{\aaa} & P_{\kk_1}\ar[rd]|{\uu}  &   \\
P_j\ar[r]|{\dd}\ar[dr]|{\vv}&U_{j}' \ar[rr]|(0.3){f_1'}\ar[ur]^{f_1}    &   & V_{j}' \ar[r]|{\cc}&P_j  \\
&P_{\kk_2}\ar[r]|{\bb} & U_{\kk_2}\ar[r]|{\bb^*} \ar[ur]|{f_2'}\ar[uu]|(0.4){f_2}   & P_{\kk_2}^*\ar[ru]|{{\vv^*}} &  }
\end{eqnarray*}
where $\vv^*={}_v(v^*)$ for $\{v\in Q_1\ |\ s(v)=j,e(v)\in\kk\}$. 

Then we give the following proposition. 

\begin{prop}\label{ver i}
For a vertex $j$ in $Q_0$ with $j\notin \kk$, we have the following right 2-almost split sequence in $\add T$
\begin{equation*}\label{simple complex}
{\displaystyle
P_{\kk_1}^*
\oplus U_{j}'
\oplus U_{\kk_2}
\xrightarrow{\left(\begin{smallmatrix}\aaa^*&0&0\\ 
f_1&f_1'&0\\
f_2&f_2'&{\bb}^*
\end{smallmatrix}\right)}
V_{\kk_1}
\oplus V_{j}'
\oplus P_{\kk_2}^*
\xrightarrow{\left(\begin{smallmatrix}\aaa\uu\\ 
\cc\\
{\vv^*}
\end{smallmatrix}\right)}{P_j.}}\end{equation*}
\end{prop}

Before starting to prove Propositions \ref{ver l}, \ref{ver i}, we prove Theorem \ref{main1} by using them.

\begin{proof}[Proof of Theorem \ref{main1}]
It is enough to show  
\begin{eqnarray}\label{isomorphic}
\PP(Q',W')\cong\End_{\TT}(T).
\end{eqnarray}
As above, we have a $K$-algebra homomorphism $\phi':\widehat{KQ'} \to \End_{\TT}(T)$. 
To show (\ref{isomorphic}), it is enough to show that $\phi'$ is surjective and $\Ker\phi'=\overline{\langle \partial_a W' \mid a \in Q'_1 \rangle}$ by definition of  the Jacobian algebra. 

Put $\Gamma:=\End_{\TT}(T)$. 
Then, by Proposition \ref{birs2}, it is enough to show that 
the following sequence is exact for any $i\in Q_0'$
\[\xymatrix@C50pt{  {\displaystyle\bigoplus_{\begin{smallmatrix}r\in \{\partial_a W'\mid a \in Q'_1\}, e(r)=i\end{smallmatrix}}}\Gamma(\phi s(r)) \ar[r]^{\ \ \ \ \ {}_r(\phi\partial_a^\mathbb{R}r)_{a}}  & 
{\displaystyle\bigoplus_{\begin{smallmatrix}a\in Q_1', e(a)=i\end{smallmatrix}}} \Gamma(\phi s(a)) \ar[r]^{\ \ \ \ \  \ {}_a(\phi a)} &  J_\Gamma(\displaystyle{\phi i}) \ar[r]& 0.  }\]
Then, by applying $\Hom_{\TT}(T,-)$ to the complexes of Propositions \ref{ver l} and \ref{ver i}, 
we have $\Gamma$-module exact sequences from definition right 2-almost split sequences. 
By expressing them in terms of $Q'$ and $W'$, one can check that they are the desired sequences. 
\end{proof}

The rest of this paper is devoted to showing Propositions \ref{ver l} and \ref{ver i}.

%%%%%%%%%%%%%%%%%%%%%%%%%%%%%%%%%%%%%%%%%%%%%%%%%%%%%%%%%%%%%%%%%%%%%%%%%%%%%%%%%%%%%%%%%%%%%%

\subsection{Exactness of some sequences}
We keep the notation of previous subsections. 
We will show the following proposition.

\begin{prop}\label{I-III}

\begin{itemize}
\item[(a)] For any $l\in\kk$, we have the following weak 2-almost split sequence in $\add T$:
$$P_l^{\ast}\xrightarrow{g_l}V_{l}\xrightarrow{f_{l0}f_{l2}}U_{l}\xrightarrow{h_l}P_l^{\ast},$$
where $g_l,f_{l0},f_{l2}$ and $h_l$ are the morphism given in $(\ref{resolution}), (\ref{h_i}), (\ref{g_i})$.
\item[(b)] For any $l,m\in\kk$, the following sequence is exact:
$$
\TT(P_m^{\ast},P_l^{\ast})\xrightarrow{g_l}\TT(P_m^{\ast},V_{l})\xrightarrow{f_{l0}}\TT(P_m^{\ast},P_l).
$$
Therefore the following sequence is exact:
$$
\TT(T,P_{\kk_1}^{\ast})\xrightarrow{\aaa^*}\TT(T,V_{\kk_1})\xrightarrow{\aaa}\TT(T,P_{\kk_1}).
$$
\item[(c)] For any $j\in Q_0$ with $j\notin\kk$ and any $l\in\kk$, the following sequence is exact:
$$
\TT(P_l^{\ast},U_{j})\xrightarrow{f_{j1}}\TT(P_l^{\ast},V_{j})\xrightarrow{f_{j0}}\TT(P_l^{\ast},P_j).
$$
Therefore the following sequence is exact:
$$
\TT(T,U_{j})\xrightarrow{f_{j1}}\TT(T,V_{j})\xrightarrow{f_{j0}}\TT(T,P_j).
$$
\end{itemize}
\end{prop}

We prepare the following notations. 
For any $l\in\kk$, we denote by $C_l:=\Coker f_{l2}$. Then we have the following exact sequence 
\begin{eqnarray*}
 P_l\xrightarrow{f_{l2}}U_{l}\xrightarrow{q_l}C_l\to0\ \ \mbox{ and }\ \ 0\xrightarrow{}C_{l}\xrightarrow{r_l}V_{l}\xrightarrow{f_{l0}}P_{l}.
\end{eqnarray*}

First we give the following easy observation.

\begin{lemm}\label{obs}
Let $p=(p^i)_{i\in\mathbb{Z}}\in {\TT}(P_l^*, P_m^*)$ be a morphism for $l,m\in \kk$. Then we have the following commutative diagram.
\[
\xymatrix@C30pt@R30pt{0\ar[r]& \Soc P_l\ar[r]^{d_{l}} \ar[d]^{p^{-2}} & P_l \ar[r]^{f_{l2}}\ar[d]^{p^{-1}} & U_{l}\ar[r]^{q_{l}}\ar[d]^{p^0} & C_l\ar[r]\ar[d]^{p^{1}}& 0\\
0\ar[r]& \Soc P_m\ar[r]_{d_{m}}& P_m \ar[r]_{f_{m2}}& U_{m}\ar[r]_{q_{m}} & C_m\ar[r]& 0.   }
\] 
If $p\in J_{\TT}(P_l^*, P_m^*)$, then $p^{-2}=0$, and 
there exist $j_{lm}^0\in\Hom_\Lambda(U_{l},P_m)$ and $j_{lm}^1\in\Hom_\Lambda(C_l,U_{m})$ such that $p^{-1}=f_{l2}j_{lm}^0$, $p^0=j_{lm}^0f_{m2}+q_lj_{lm}^1$ and $p^{1}=j_{lm}^1q_m$.
\end{lemm}

\begin{proof}
The first assertion is clear. Assume that $p\in J_{\TT}(P_l^*, P_m^*)$. 
If $p^{-2}$ is isomorphic, then it implies that $p^{-1}$ is isomorphic since $d_l$ is  
an injective hull. 
Since $f_{l2}$ and $f_{m2}$ are minimal left $(\add(\La/P_l))$-approximation, 
$p^0$ is also isomorphic, a contradiction to $p\in J_{\TT}(P_l^*, P_m^*)$. 
Thus $p^{-2}$ is not an isomorphism.
Because $\Soc P_l,\ \Soc P_m$ are simple modules, we have $p^{-2}=0$ .
Therefore we obtain $d_lp^{-1}=p^{-2}d_m=0.$ 
Since $P_m$ is an injective module, there exists $j_{lm}^0\in\Hom_\Lambda(U_{l},P_m)$ such that $p^{-1}=f_{l2}j_{lm}^0$. 
Then since $U_{m}$ is an injective module and $f_{l2}(p^0-j_{lm}^0f_{m2})=0$, there exists $j_{lm}^1\in\Hom_\Lambda(C_l,U_{m})$ such that 
$q_lj_{lm}^1=p^0-j_{lm}^0f_{m2}$. 
Since $q_l$ is surjective, we have $p^{1}=j_{lm}^1q_m$.
\end{proof}

Now we are ready to prove Proposition \ref{I-III}.

\begin{proof}
(a) (i) We will show that $h_l$ is right almost split in $\add T$. 

Take any morphism $p=(p^i)_{i\in\mathbb{Z}}\in J_{\TT}(\Lambda/P_{\kk},P_l^*)$. 
Then clearly $p^0\in\Hom_\Lambda(\Lambda/P_{\kk},V_{l})$ gives a morphism 
$g=(g^i)_{i\in\mathbb{Z}}\in\TT(\Lambda/P_{\kk},V_{l})$ by $g^0=p^0$ and $g^i=0$ for $i\neq0$. Thus we have $gh_l=p$.

Next for any $m\in\kk$, take any morphism $p=(p^i)_{i\in\mathbb{Z}}\in J_{\TT}(P_{m}^*,P_l^*)$. 
By Lemma \ref{obs}, there exists $j^0_{ml}\in\Hom_\Lambda(U_{m},P_l)$ such that 
$p^{-1}=f_{m2}j_{ml}^0$. 
Then the morphism $p^0-j_{ml}^0f_{l2}\in\Hom_\Lambda(U_{m},U_{l})$ gives a morphism 
$g\in \TT(P_m^*,U_{l})$ satisfying $p=gh_l$.

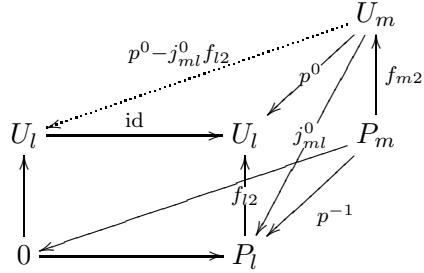
\[
\xymatrix@C30pt@R30pt{
&&&U_{m}\ar@{.>}[llld]_{p^0-j_{ml}^0f_{l2}}\ar[ld]|{p^0}\ar[ldd]|{j_{ml}^0}\\
 U_{l}\ar[rr]^{{\rm id}}  && U_{l} &P_m\ar[u]_{f_{m2}}\ar[ld]^{p^{-1}}\ar[llld]_(0.6){} \\
0\ar[rr]\ar[u]& &P_l\ar[u]|{f_{l2}}&   }
\]

(ii) We will show that $f_{l0}f_{l2}$ is a pseudo-kernel of $h_l$ in $\add T$.

Since $(\ref{h_i})$ is a triangle, we have an exact sequence
$\TT(T,P_l)\xrightarrow{f_{l2}}\TT(T, U_{l})\xrightarrow{h_{l}}\TT(T, P_l^*)$. 
Thus we only have to show that 
$\TT(T,V_{l})\xrightarrow{f_{l0}}\TT(T,P_l)$ is surjective. 

Take any morphism $p=(p^i)_{i\in\mathbb{Z}}\in \TT(\Lambda/P_{\kk},P_l)$. 
Then since $f_{l0}$ is a right $(\add(\Lambda/P_l))$-approximation, there exists 
$g^0\in \Hom_\Lambda(\Lambda/P_{\kk},V_{l})$ such that $p^0=g^0f_{l0}$. 
Thus $g^0$ gives a morphism $g\in\TT(\Lambda/P_{\kk},V_{l})$ satisfying   
$p=gf_{l0}$.

Next for any $m\in\kk$, take any morphism $p=(p^i)_{i\in\mathbb{Z}}\in \TT(P_{m}^*,P_l)$. 
Then, by $f_{m2}p^0=0$,
there exists $s\in\Hom_\Lambda(C_{m},P_l)$ such that $p^0=q_ms$.

Then since $P_l$ is an injective module, there exists $t\in\Hom_\Lambda(V_{m},P_l)$ such that 
$s=r_mt$. Moreover, by the assumption (a2), we have $P_l\notin\add V_{m}$.
Then since $f_{l0}$ is a right $(\add(\Lambda/P_l))$-approximation, there exists 
$u\in\Hom_\Lambda(V_{m},V_{l})$ such that $t=uf_{l0}$. 
Thus we have $p^0=(q_mr_mu)f_{l0}$, and $g^0:=q_mr_mu$ gives a morphism $g\in\TT(P_{m}^*,V_{l})$ satisfying $p=gf_{l0}$.

\[
\xymatrix@C30pt@R30pt{&&V_{m}\ar@{.>}[dd]|t\ar@{.>}[lldd]|u&C_m\ar[l]^{r_m}\ar@{.>}[ldd]|s\\
&&&U_{m}\ar[u]_{q_m}\ar[ld]|{p^0}\ar@{-->}[llld]|{q_mr_mu}\\
 V_{l}\ar[rr]^{f_{l0}}  && P_{l} &P_m\ar[u]_{f_{m2}}\ar[ld]^{}\ar[llld] \\
0\ar[rr]\ar[u]& &0\ar[u]|{}&   }
\]

(iii) We will show that $g_l$ is left almost split in $\add T$. 

Take any morphism $p=(p^i)_{i\in\mathbb{Z}}\in J_{\TT}(P_l^*,\Lambda/P_{\kk})$.
Then by $f_{l2}p^0=0$,
there exists $s\in\Hom_\Lambda(C_{l},\Lambda/P_{\kk})$ such that $p^0=q_ls$.
Then since $\Lambda/P_{\kk}$ is an injective module and $r_l$ is injective, there exists $t\in\Hom_\Lambda(V_{l},\Lambda/P_{\kk})$ such that $s=r_lt$.
Then we have $p^0=q_lr_lt=f_{l1}t$ and $t$ gives a morphism $t\in\TT(V_{l},\Lambda/P_{\kk})$ satisfying 
$p=g_{l}t$.

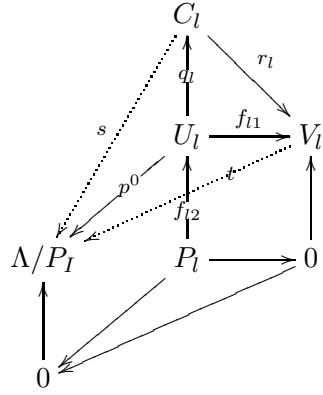
\[
\xymatrix@C30pt@R30pt{
&&C_l\ar[rd]^{r_l}\ar@{.>}[ldd]_s&&\\
&&U_{l}\ar[u]|{q_l}\ar[r]^{f_{l1}}\ar[ld]|(0.4){p^0}&V_{l}\ar@{.>}[lld]|(0.3)t&\\
&\La/P_{\kk}&P_l\ar[ld]\ar[u]|(0.4){f_{l2}}\ar[r]&0\ar[u]\ar[lld]&\\
&0\ar[u]&&&  \\
&&&&}
\]

Next for any $m\in\kk$, take any morphism $p=(p^i)_{i\in\mathbb{Z}}\in J_{\TT}(P_l^*,P_m^*)$. 
Then, by Lemma \ref{obs}, we have ${j}_{lm}^0\in \Hom_\Lambda(U_{l}, P_m)$ and ${j}_{lm}^1\in \Hom_\Lambda(C_{l}, U_{m})$ such that 
$p^0={j}_{lm}^0f_{m2}+q_l{j}_{lm}^1$ and $p^{-1}=f_{l2}{j}_{lm}^0$. 
By the same argument of the first case, there exists $t\in\Hom_\Lambda(V_{l},U_{m})$ such that $p^0={j}_{lm}^0f_{m2}+f_{l1}t$ and $t$ gives a morphism $t\in\TT(V_{l},P_m^*)$ satisfying $p=g_{l}t$.

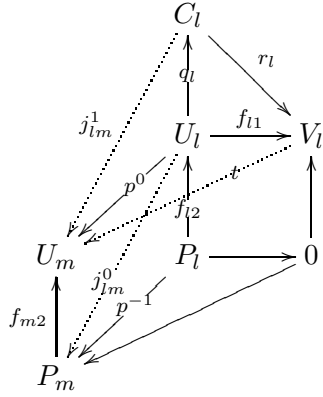
\[
\xymatrix@C30pt@R30pt{
&&C_l\ar[rd]^{r_l}\ar@{.>}[ldd]_{j_{lm}^1}&&\\
&&U_{l}\ar[u]|{q_l}\ar[r]^{f_{l1}}\ar@{.>}[ldd]|(0.6){j_{lm}^0}\ar[ld]|(0.4){p^0}&V_{l}\ar@{.>}[lld]|(0.3)t&\\
&U_{m}&P_l\ar[ld]|(0.4){p^{-1}}\ar[u]|(0.4){f_{l2}}\ar[r]&0\ar[lld]\ar[u]&\\
&P_m\ar[u]^{f_{m2}}&&&  \\
&&&&}
\]

(iv) We will show that $f_{l0}f_{l2}$ is a pseudo-cokernel of $g_l$ in $\add T$.

Assume $p=(p^i)_{i\in\mathbb{Z}}\in \TT(V_{l},\Lambda/P_{\kk})$ satisfies 
$g_lp=0$.
Then since $f_{l1}p^0=0$, there exists $s\in\Hom_\Lambda(P_{l},\Lambda/P_{\kk})$ such that $p^0=f_{l0}s$.

\[
\xymatrix@C30pt@R30pt{
&U_{l}\ar[r]^{f_{l1}}& V_{l}\ar[lld]|(0.3){p^0}\ar[r]^{f_{l0}}  & P_l\ar@{.>}[llld]^s \\
\Lambda/P_{\kk}&P_l\ar[r]\ar[u]^{}&0\ar[u]\ar[lld]\ar[r]\ar[u]& 0\ar[u]^{} \\  
0\ar[u]&&&}\]

Then since $f_{l2}:P_l\to V_{l}$ is a left $(\add(\Lambda/P_l))$-approximation, 
there exists $t\in\Hom_\Lambda(U_{l},\Lambda/P_{\kk})$ such that $s=f_{l2}t$. 
Thus we have $p^0=(f_{l0}f_{l2})t$ and 
$t$ gives a morphism $t\in\TT(U_{l},\Lambda/P_{\kk})$ satisfying $p=(f_{l0}f_{l2})t$.

Next for any $m\in\kk$, assume that $p=(p^i)_{i\in\mathbb{Z}}\in \TT(V_{l},P_{m}^*)$ satisfies $g_lp=0$. 
Then there exists $u\in \Hom_\Lambda(U_{l},P_m)$ such that $f_{l1}p^0=uf_{m2}$. 

Since we have $f_{l2}u=0$, 
there exists $u'\in\Hom_\La(C_l,P_m)$ such that $u=q_lu'$. 
Since $P_{m}$ is an injective module and $r_l$ is injective, 
there exists $u''\in\Hom_\Lambda(V_{l},P_{m})$ such that $u'=r_lu''$. 
Hence we have $f_{l1}p^0-uf_{m2}=f_{l1}p^0-q_lr_lu''f_{m2}=
f_{l1}(p^0-u''f_{m2})=0.$

Thus there exists $s\in\Hom_\Lambda(P_{l},U_{m})$ such that $p^0-u''f_{m2}=f_{l0}s$.
By the assumption (a2), we have $P_l\notin\add U_{m}$.
Then since $f_{l2}:P_l\to U_{l}$ is a left $(\add(\Lambda/P_l))$-approximation, 
there exists  
$t\in\Hom_\Lambda(U_{l},U_{m})$ such that $p^0-u''f_{m2}=(f_{l0}f_{l2})t$. Then  
$t$ gives a morphism $t\in\TT(U_{l},U_{m})$ satisfying $p=(f_{l0}f_{l2})t$. 
 
\[\xymatrix@C30pt@R30pt{&C_l\ar@{.>}[lddd]|(0.3){u'}\ar[rd]^{r_l}&&\\
&U_{l}\ar[r]^{f_{l1}}\ar[u]|{q_l}\ar[ldd]|{u}& V_{l}\ar[lld]|(0.3){p^0}\ar[r]^{f_{l0}}\ar@{.>}[lldd]|(0.7){u''}  & P_l\ar@{.>}[llld]|(0.4)s \\
U_{m}&P_l\ar[r]\ar[u]|(0.3){f_{l2}}&0\ar[u]\ar[lld]\ar[r]\ar[u]& 0\ar[u]^{} \\  
P_m\ar[u]^{f_{m2}}&&&}\]

%%%%%%%%%%%%%%%%%%%%%%%%%%%%%%%%%%%%%%
(b) 
Assume that $p=(p^i)_{i\in\mathbb{Z}}\in \TT(P_m^*,V_{l})$ satisfies 
$pf_{l0}=0$. 
Then since $p^0f_{l0}=0$, there exists
$h^0\in\Hom_\Lambda(U_{m},U_{l})$ such that $p^0=h^0f_{l1}$. 
Since $f_{m2}h^0f_{l1}=f_{m2}p^0=0,$ there exists
$h^{-1}\in\Hom_\Lambda(P_{m},P_l)$ such that $f_{m2}h^0=h^{-1}f_{l2}$.
Then $h^0,h^{-1}$ give a morphism $h\in\TT(P_m^*,P_l^*)$ satisfying $p=hg_l$.

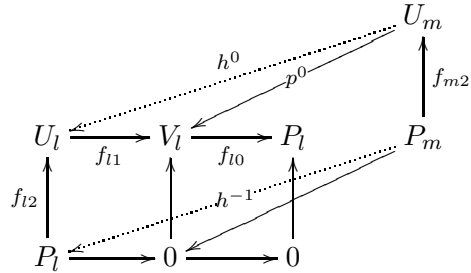
\[
\xymatrix@C30pt@R30pt{
&&&U_{m}\ar[lld]|{p^0}\ar@{.>}[llld]_(0.5){h^0}\\
U_{l}\ar[r]_{f_{l1}}& V_{l}\ar[r]_{f_{l0}}  & P_l &P_m\ar[u]_{f_{m2}}\ar@{.>}[llld]|{h^{-1}}\ar[lld] \\
P_l\ar[r]\ar[u]^{f_{l2}}&0\ar[r]\ar[u]& 0\ar[u]^{}&   }
\]

Moreover, it is easy to see that the following sequence is exact by (\ref{resolution}) 
$$\TT(\La/P_{\kk},P_l^{\ast})\xrightarrow{g_l}\TT(\La/P_{\kk},V_{l})\xrightarrow{f_{l0}}\TT(\La/P_{\kk},P_l).$$
Hence the second statement follows immediately from the first one.
%%%%%%%%%%%%%%%%%%%%%%%%%%%%%%%%%%%%%%

(c)
Assume that $p=(p^i)_{i\in\mathbb{Z}}\in \TT(P_l^*,V_{j})$ satisfies 
$pf_{j0}=0$.  
Then there exists $h^0\in\Hom_\Lambda(U_{l},U_{j})$ such that $p^0=h^0f_{j1}$. 
Moreover, since $f_{l2}h^0f_{j1}=f_{l2}p^0=0$, there exists 
$h^{-1}\in\Hom_\Lambda(P_l,P_j)$ such that $f_{l2}h^0=h^{-1}f_{j2}.$
Since $l\neq j$, we have $h\in J_{\TT}(P_l^*,P_j^*)$. 
Then by Lemma \ref{obs}, there exists $j_{lj}^0\in \Hom_\Lambda(U_{l},P_j)$ such that 
$h^{-1}=f_{l2}{j}_{lj}^0.$

\[
\xymatrix@C30pt@R30pt{
&&&U_{l}\ar[lld]|{p^0} \ar@{.>}[llld]_(0.5){h^0}\ar@{.>}[llldd]|(0.8){j^0_{lj}}\\
U_{j}\ar[r]_{f_{j1}}& V_{j}\ar[r]_{f_{j0}}  & P_j &P_l\ar[u]_{f_{l2}}\ar@{.>}[llld]|{h^{-1}}\ar[lld] \\
P_j\ar[r]\ar[u]^{f_{j2}}&0\ar[r]\ar[u]& 0\ar[u]^{}&   }
\]

Then $h^0-j^0_{lj}f_{j2}\in\Hom_\Lambda(U_{l},U_{j})$ gives a morphism 
$h\in\TT(P_l^*,U_{j})$ satisfying $p=hf_{j1}$.  
   
\[
\xymatrix@C30pt@R30pt{
&&&U_{l}\ar[lld]|{p^0}\ar[llld]_(0.5){h^0-j^0_{lj}f_{j2}}\\
U_{j}\ar[r]_{f_{j1}}& V_{j}\ar[r]_{f_{j0}}  & P_j&P_l\ar[u]_{f_{l2}}\ar[llld]\ar[lld] \\
0\ar[r]\ar[u]&0\ar[r]\ar[u]& 0\ar[u]^{}&   }
\]

The second statement follows immediately from the first one.
\end{proof}

%%%%%%%%%%%%%%%%%%%%%%%%%%%%%%%%%%%%%%%%%%%%%%%%%%%%%%%%%%%%%%%%%%%%%%%%%%%%%%%%%%%%%%%%%%%%%%

\subsection{Proof of Propositions \ref{ver l} and \ref{ver i}}

Then we give the proof of Propositions \ref{ver l} and \ref{ver i}.

\begin{proof}
Proposition \ref{ver l} immediately follows from Proposition \ref{I-III} (a). 
We will show Proposition \ref{ver i}.
Since we have 
$\bb^*\vv^*+(f_2\ f_2'){\aaa\uu\choose \cc}={}_{b,v}(b^*v^*+(\partial_{[vb]}W))=0$ by the definition of the algebra homomorphism, it is a complex.

\emph{Step 1}. We will show that $\left(\begin{smallmatrix}\aaa\uu\\ 
\cc\\
{\vv^*}
\end{smallmatrix}\right)$ is right almost split in $\add T$.

(i) First we will show that any morphism 
$p\in J_{\TT}(\La/P_{\kk},P_j)$ factors through ${\aaa\uu\choose \cc}$. 
Since ${\uu\choose \cc}:P_{\kk_1}\oplus V_{j}'{\to} P_j$ is right almost split in $\add\La$, 
there exists $(p_1\ p_2)\in\TT(\La/P_{\kk},P_{\kk_1}\oplus V_{j}')$ such that $p=p_1{\uu}+p_2\cc$.
Since $\aaa:V_{\kk_1}\xrightarrow{} P_{\kk_1}$ is a right $(\add(\Lambda/P_{\kk_1}))$-approximation, 
there exists $q\in\TT(\La/P_{\kk},V_{\kk_1})$ such that $p_1=q\aaa$.
Thus we have $p=(q\ p_2){\aaa \uu\choose \cc}$.
\[\xymatrix@C=2cm{
&&\La/P_{\kk}\ar^p[dd]\ar_{p_1}[dl]\ar_{p_2}[ddl]\ar_q@/_3pc/[dll]\\
V_{\kk_1}\ar[r]^{\aaa}&P_{\kk_1}\ar[dr]^{\uu}&\\
&V_{j}'\ar_{\cc}[r]&P_j
}\]

(ii) Next we take any $p\in J_{\TT}(P_l^{\ast},P_j)$ for $l\in\kk$ with $l\notin \kk_2$. 
Note that there is no arrow $j\to l$ in $Q$ in this case. 
Since $g_l$ is left almost split in $\add T$ by Proposition \ref{I-III} (a), there exists
$q\in\TT(V_{l},P_j)$ such that $p=g_lq$.
Since $P_l\notin\add U_{j}$, we have $P_j\notin\add V_{l}$.
Thus $q\in J_{\TT}(V_{l},P_j)$ holds.
Moreover, by the assumption ($a2$), we have $P_{\kk_1}\notin \add V_{l}$. 
Then the first case implies that
$q$ factors through ${\aaa\uu\choose \cc}$. 
\[\xymatrix@C=2cm{
P_l^{\ast}\ar_p[d]\ar^{g_l}[r]&V_{l}\ar^q[dl]\\
P_j}\]

(iii) Take $l\in \kk_2$ and decompose $P_{\kk_2}^*=({P_l^*})^{n_l}\oplus X$ such that $X\notin\add {P_l^*}$.
We will show that the map
\[(v^*)_{\{v\in Q | v:j\to l\}}\colon (\TT/J_{\TT})(P_l^{*},({P_l^*})^{n_l})\to (J_{\TT}/J_{\add T}^2)(P_l^*,P_j)\]
is bijective.
Since $\TT$ is a Krull-schmidt category and $P_l^*$ is indecomposable, we have $K=(\TT/J_{\TT})(P_l^{\ast},{P_l^{\ast}})$.
On the other hand, by Proposition \ref{I-III} (a), we have that $g_l:P_l^* \to V_{l}=\oplus_{\{a\in Q,e(a)=l\}}P_{s(a)}$ is minimal left almost split in $\add T$
since the middle morphism $f_{l0}f_{l2}$ in the sequence of Proposition \ref{I-III} (a) belongs to $J_{\TT}$.
Thus we have that
$(J_{\TT}/J_{\add T}^2)(P_l^*,P_j)$ is a $K$-vector space with basis $\{v^*\ |\ v\in Q_1; v : j\to l \}$. Thus the above map is bijective.

%[[ikara dete l ni hairu yano kazu]=[l kara deru irreducible map no kazu]$$P_l^*\to V_{l}=\oplus_{\{v\in Q,e(v)=l\}}P_{s(v)}\to P_l.]]$$

Then take any $p\in J_{\TT}(P_l^{\ast},P_j)$ for $l\in \kk_2$. 
Let $(v^*)_v:({P_l^*})^{n_l}\to P_i$ be a restriction of $\vv^{*}$. 
By the above bijection, there exists $p_1\in\TT(P_l^{\ast},({P_l^*})^{n_l})$ such that $p-p_1(v^*)_v\in J_{\add T}^2(P_l^*,P_j)$.
Since $g_l$ is right almost split in $\add T$ by (a), 
there exists $q\in J_{\TT}(V_{l},P_j)$ such that $p-p_1(v^*)_v=g_lq$. 
Then, by the same argument of (i) and (ii), $q$ factors through ${\aaa\uu\choose \cc}$.
Thus $p$ factors through $\left(\begin{smallmatrix}\aaa\uu\\ 
\cc\\
{\vv^*}
\end{smallmatrix}\right).$

\[\xymatrix@C=2cm{
V_{\kk_1}\ar[r]^{\aaa}&P_{\kk_1}\ar[dr]^{\uu}&P_l^*\ar^p[d] \ar[ddl]|(0.2){p_1}\ar[r]^{g_l}&V_{l}\ar[dl]^{q}\\
&V_{j}'\ar[r]|(0.3){\cc}&P_j\\
&P_{\kk_2}^*\ar[ur]_{\vv^*}&}\]

\emph{Step 2}. We will show that $\left(\begin{smallmatrix}\aaa^*&0&0\\ 
f_1&f_1'&0\\
f_2&f_2'&{\bb}^*
\end{smallmatrix}\right)$ is a pseudo-kernel of $\left(\begin{smallmatrix}\aaa\uu\\ 
\cc\\
{\vv^*}
\end{smallmatrix}\right)$ in $\add T$.

Assume that $(p_1\ p_2\ p_3)\in\TT(T ,V_{\kk_1}\oplus V_{j}'\oplus P_{\kk_2}^*)$ satisfies $(p_1\ p_2\ p_3)\left(\begin{smallmatrix}\aaa\uu\\ 
\cc\\
{\vv^*}
\end{smallmatrix}\right)=0$.
We first show that there exists $q_1\in\TT(T,U_{\kk_2})$ such that $p_3=q_1\bb^*$.
Since $\bb^*$ is right almost split in $\add T$ by Proposition \ref{I-III} (a), we only have to show $p_3\in J_{\TT}$. 
We only have to consider the case $T=P_l^{\ast}$ for $l\in \kk_2$. 
Then since $p_3\vv^*=-p_1\aaa\uu-p_2\cc\in J_{\add T }^2$, we have $p_3\in J_{\TT}$.

Then we have 
\begin{eqnarray*}
((p_1-q_1f_2)\aaa\ \ p_2-q_1f_2'){ \uu\choose \cc}&=&({p_1\ p_2}){\aaa\uu\choose \cc}-q_1({f_2\ f_2'}){\aaa\uu\choose \cc}\\
&=&({p_1\ p_2}){\aaa\uu\choose \cc}+q_1\bb^*\vv^{*}\\
&=&({p_1\ p_2}){\aaa\uu\choose \cc}+p_3\vv^{*}\\
&=&0.
\end{eqnarray*}

Therefore, by Proposition \ref{I-III} (c),   
there exists $(g_1\ g_2)\in\TT(T,U_{j}'\oplus P_{\kk_2})$ such that 
$$(g_1\ g_2){f_1\aaa\ \ f_1' \choose \bb f_2\aaa \ \ \bb f_{2}'}=((p_1-q_1f_2)\aaa\ \ p_2-q_1f_2').$$
Thus we have 
$(p_1\aaa\ p_2\ p_3)=(g_1\ q_2){f_1\aaa\ \ f_1' \ \ \ 0\choose f_2\aaa\ \ f_2'\ \ \bb^*},$ 
where we put $q_2=q_1+g_2\bb$.

Moreover, since we have $(p_1-(g_1f_1+q_2f_2))\aaa=0$,  
there exists $q_3\in\TT(T,P_{\kk_1}^*)$ such that $q_3\aaa^*=p_1-(g_1f_1+q_2f_2)$ by Proposition \ref{I-III} (b). 
Thus we have 
$$(p_1\ p_2\ p_3)=(q_3\ g_1\ q_2)\left(\begin{smallmatrix}\aaa^*&0&0\\ 
f_1&f_1'&0\\
f_2&f_2'&{\bb}^*
\end{smallmatrix}\right).
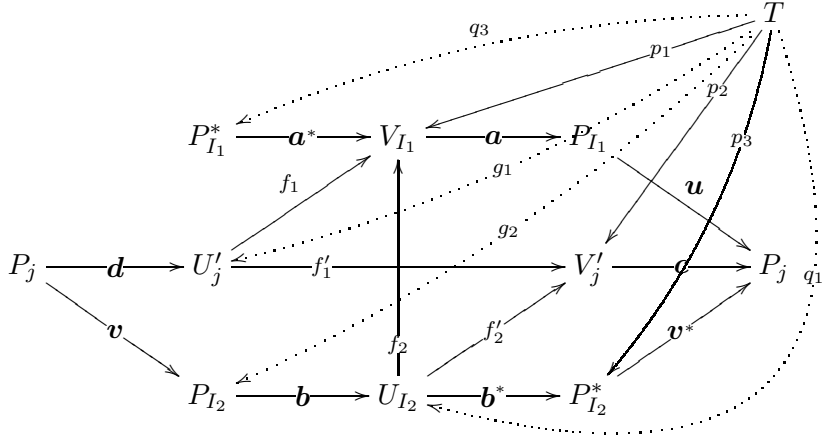
$$
\[\xymatrix@C50pt@R30pt{&&&&T\ar|(0.3){p_1}[lld]\ar|(0.3){p_2}[ldd]\ar|(0.3){p_3}@/^1pc/[lddd]\ar@{.>}_{q_3}@/_1pc/[llld]\ar@{.>}|{g_1}@/^1pc/[llldd]\ar@{.>}|{g_2}@/^1pc/[lllddd]\ar@{.>}|(0.3){q_1}@/^8pc/[llddd]\\
&P_{\kk_1}^* \ar[r]|{\aaa^*}  &V_{\kk_1} \ar[r]|{\aaa} & P_{\kk_1} \ar[rd]^{\uu}  &   \\
P_j\ar[r]|{\dd}\ar[dr]|{\vv}&U_{j}'\ar[rr]|(0.3){f_1'}\ar[ur]^{f_1}    &   & V_{j}'  \ar[r]|{\cc}&P_j  \\
&P_{\kk_2}\ar[r]|{\bb} & U_{\kk_2}\ar[r]|{\bb^*} \ar[ur]|{f_2'}\ar[uu]|(0.2){f_2}   & P_{\kk_2}^* \ar[ru]|{{\vv^*}} &  }
\] 

\end{proof}

\subsection*{Acknowledgement}
The author is supported by Grant-in-Aid
for JSPS Fellowships No.23.5593. 
He would like to thank Osamu Iyama for his support and patient guidance. He would like to thank Hideto Asashiba for stimulating discussions and questions. He is grateful to Martin Herschend, who kindly explain the construction of selfinjective QPs, and Kota Yamaura and Takahide Adachi for their valuable comments and advice. 
The author is very grateful to the referees for the valuable comments and suggestions, which improve the presentation.

%%%%%%%%%%%%%%%%%%%%%%%%%%%%%%%%%%%%%%%%%%%%%%%%%%%%%%%%%%%%%%%%%%%%%%%%%%%%%%%%%%%%%%%

\end{document}